\documentclass[12pt,a4paper]{amsart}
\usepackage{amssymb,amsmath,amsthm}
\usepackage{graphicx}
\usepackage{xcolor}

\usepackage{epsfig}

\long\def\onefigure#1#2{
\begin{figure*}[tbp]
\begin{center}
#1
\end{center}
\caption{#2}
\end{figure*}
}

\newcommand{\lipefig}[2]  
{\onefigure{\mbox{\psfig{file=#1.eps}}}{\label{f:#1} #2} }

\newtheorem{theorem}{Theorem}[section]
\newtheorem{lemma}{Lemma}[section]

\newtheorem{claim}{Claim}[section]
\newtheorem{fact}{Fact}[section]
\newtheorem{defn}{Definition}[section]

\newcommand{\al}{\alpha}

\newcommand{\la}{\lambda}

\newcommand{\R}{\mathbb{R}}

\newcommand{\T}{\mathcal{T}}
\newcommand{\F}{\mathcal{F}}

\newcommand{\Ave}{\textrm{Ave}}
\newcommand{\diam}{\textrm{diam\;}}
\newcommand{\lin}{\textrm{lin}}

\newcommand{\remove}[1]{}

\numberwithin{equation}{section}

\begin{document}

\title{Same average in every direction}
\author{Imre B\'ar\'any and G\'abor Domokos}
\keywords{Convex polytopes, zonotopes, number of vertices, fragmentation}
\subjclass[2020]{Primary 52A15, secondary 52A20}

\numberwithin{equation}{section}

\maketitle

\begin{abstract} Given a polytope $P\subset \R^3$ and a non-zero vector $z \in \R^3$, the plane $\{x\in \R^3:zx=t\}$ intersects $P$ in a convex polygon $P(z,t)$ for $t \in [t^-,t^+]$ where $t^-=\min \{zx: x \in P\}$ and $t^+=\max \{zx: x\in P\}$, $zx$ is the scalar product of $z,x \in \R^3$. Let $A(P,z)$ denote the average number of vertices of $P(z,t)$ on the interval $[t^-,t^+]$. For what polytopes is $A(P,z)$ a constant independent of $z$?
\end{abstract}

\section{Introduction}\label{sec:introd}

Assume $P$ is a convex polytope  and $z$ is a non-zero vector in 3-space. Set $t^-=t^-(P,z)=\min \{zx: x \in P\}$ and $t^+=t^+(P,z)=\max \{xz: x\in P\}$, where  $zx$ is the scalar product of $z,x \in \R^3$. For $t \in [t^-,t^+]$ the intersection of $P$ with the plane $\{x\in \R^3:zx=t\}$ is a convex polygon $P(z,t)$. Define $A(P,z)$ as the average of the number of vertices of $P(z,t)$ on the interval  $[t^-,t^+]$, that is
\begin{equation}\label{eq:ave}
A(P,z)=\frac 1{t^+-t^-}\int_{t^-}^{t^+} \mbox{number of vertices of }P(z,t)dt.
\end{equation}

Note that in this definition we may assume that $z$ is a unit vector because $A(P,z)=A(P,sz)$ for every non-zero real number $s$. In this case $t^+-t^-=w(P,z)$, the width of $P$ in direction $z$, otherwise $t^+-t^-=\|z\|w(P,z)$ where $\|z\|$ is the Euclidean norm of $z\ne 0.$

\smallskip
The quantity $A(P,z)$, the average number of vertices (or of edges) of the convex  polygons $P(z,t)$, has come up in geology, and a strange phenomenon has been observed. Namely that $A(P,z)=4$ for every $z$ when $P$ is a cube in $\R^3$. Subsequently the following question emerged. For what polytopes $A(P,z)$ is a constant independent of $z$? We answer this question for centrally symmetric polytopes in Theorem~\ref{th:main} below.
We return to the connection of $A(P,z)$ to geology, in particular to rock formations in the last section of this paper.

\medskip
A zonotope $P$ is the Minkowski sum of intervals $[0,a_i]$, where $a_1,\ldots,a_n$ are non-collinear vectors (called the generators of $P$) in $\R^d$, $n\ge d$, see \cite{schn}. In particular the cube in $\R^3$ is a zonotope with 3 generators. More generally the following is true.

\begin{fact}\label{fact:zono} If $P \in \R^3$ is a zonotope with $n$ generators no three of which are collinear, then $A(P,z)=2(n-1)$ for every vector $z\ne 0.$
\end{fact}

The proof is given in Section~\ref{sec:lem11}. For the statement of our main theorem a definition is needed.
Assume $P\subset \R^3$ is a zonotope with generators $v_1,\ldots,v_n$ and with $A(P,z)$ constant, say $\la$. Define a hypergraph $\F$ with vertex set $V=\{v_1,\ldots,v_n\}$ and $W\subset V$ is an edge in $\F$ if $\dim \lin W=2$ and $W$ is maximal with this property, that is, $\dim \lin (W \cup \{v\})=3$
for every $v \in V \setminus W.$ The degree, $\deg v$, of $v \in V$ is the number of edges of $\F$ containing it.

\begin{theorem}\label{th:main} Assume $P\subset \R^3$ is a centrally symmetric polytope with generator set $V$. Then $A(P,z)=\la$ for every $z \ne 0$ if and only if $P$ is a zonotope with $2\deg v=\la$ for every $v\in V$ in the hypergraph $\F.$
\end{theorem}

This theorem shows that $\la$ is an even integer when $P$ is centrally symmetric. We are going to give examples of non-symmetric polytopes with $A(P,z)=\la$, a constant for every $z$. In these examples $\la$ is always an even integer. However this is not the case in general because there are polytopes
$P\subset \R^3$ with $A(P,z)=\la$, a constant that is not an integer. This is a result of Attila P\'or~\cite{Por}. He proves a stronger theorem, namely, that there is in open (and non-empty) interval $I \subset \R$ such that for every $\la \in I$ there is a polytope $P\subset \R^3$ with $A(P,z)=\la$ for every $z$.

\bigskip

\section{Higher dimensions}\label{sec:high}

The definition of $A(P,z)$ extends without any change to polytopes $P$ in $\R^d$. The case $d=2$ is trivial as $P(z,t)$ is a segment and $A(P,z)=2$ for every convex polygon $P$. The higher dimension version of Theorem~\ref{th:main} says the following.

\begin{theorem}\label{th:maind} Assume $P\subset \R^d$ is a centrally symmetric polytope. Then $A(P,z)=\la$ for every $z \ne 0$ if and only if $P$ is a zonotope with generator set $V$ such that the number of edges of $P$ parallel with $v$ equals $\la$ for every $v \in V.$
\end{theorem}

For the proof an auxiliary lemma is needed.

\begin{lemma}\label{l:P-P} If $P\subset \R^d$ is a $d$-dimensional polytope and $A(P,z)$ is a constant independent of $z$, then $P-P$ is a zonotope.
\end{lemma}

We are going to call a convex body $K\subset \R^d$ a {\sl half-zonotope} if $K-K$ is a zonotope. It is clear that a half-zonotope is always a polytope. Moreover
a zonotope $P$ is always a half-zonotope because, assuming that the centre of $P$ is the origin, $P-P=P+P=2P$ is clearly a zonotope. In view of
Lemma~\ref{l:P-P} one would like to see what polytopes are half-zonotopes. One case is easy:

\begin{fact}\label{f:o-sym} A centrally symmetric half-zonotope is a zonotope.
\end{fact}

The {\bf proof} is very simple: if $P$ is a $0$-symmetric half-zonotope, then $P=-P$ and $P-P=P+P=2P$ is a zonotope, so $P$ is always a zonotope.\qed

\medskip
In Section~\ref{sec:exampl} we give an example of a half-zonotope which is not a zonotope, and another non-zonotope $P\subset \R^3$ with $A(P,z)=6$ for every $z$, and more generally, for every $m\ge 2$ another non-zonotope $P$ with $A(P,z)=2m$ for every $z$.

\medskip
\section{Another expression of $A(P,z)$ and proof of Fact~\ref{fact:zono}}\label{sec:ave}

An edge of $P$ is a segment $[u,v]$ where $u,v$ are vertices of $P$ and there is a supporting hyperplane $h$ such that $P\cap h=[u,v]$. The edge is also a vector $u-v$ or $v-u$ and we can always choose (liberally)
which one. If there is another edge $e'$ parallel with $e$ and with a supporting hyperplane $h'$ also parallel with $h$, then $e$ and $e'$ are called {\sl opposite} edges. Note that the edge opposite to $e$ may not be unique. It may also happen that an edge opposite to $e$ does not exist. Then there is a vertex $w$ of $P$ and two parallel hyperplanes $h$ and $h'$ such that $P\cap h=[u,v]$ and $P\cap h'=\{w\}$, and $w$ is considered a (virtual) opposite edge to $e$.

\medskip
The edge set $E$ of $P$ is split into equivalence classes $E_1,\ldots,E_n$ by the equivalence relation ``being parallel". We derive an alternate formula for $A(P,z)$. Assume $z\in \R^d$ is a non-zero vector which is in general position with respect to $E$, that is, $ze\ne 0$ for any $e \in E$. Let $u$ and $v$ be the vertices of $P$ where the function $zx$ takes its maximum and minimum on $x \in P$, respectively. Clearly $t^+-t^-=z(u-v)$.

\medskip
Let $S$ be a two-dimensional plane, parallel with $z$, and let $\pi$ denote the orthogonal projection from $\R^d$ to $S$. Choose $S$ so that $\pi(e)$ is not a single point for any $e \in E$.  Then $\pi(P)$ is a convex polygon in $S$ and there is a path $\pi(e_1),\pi(e_2),\ldots,\pi(e_n)$, say, with $e_i \in E_i$ going along the boundary of $\pi(P)$ from $\pi(v)$ to $\pi(u)$, see Figure~\ref{fig:path}.
\begin{figure}[h]
\centering
\includegraphics[scale=0.9]{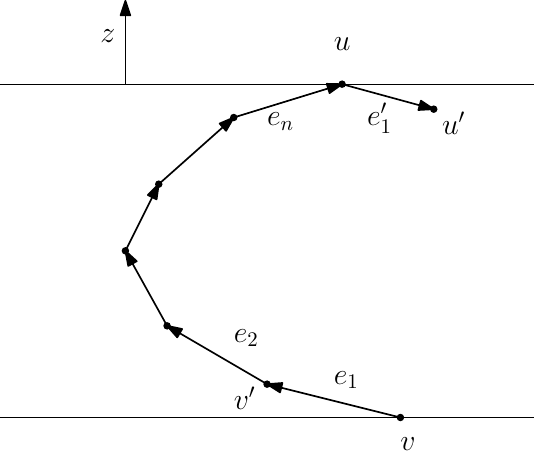}
\caption{The paths from $v$ to $u$ and from $v'$ to $u'$.}
\label{fig:path}
\end{figure}
There is at most one edge from every $E_i$ on this path, but an edge may be virtual, meaning that $e_i=0$ which is fine. The path $e_1,e_2,\ldots,e_n$ goes from $v$ to $u$ on the boundary of $P$ and $t^+(z)-t^-(z)=\sum_1^n|ze_i|$. Then
\begin{equation}\label{eq:Ave}
A(P,z)=\frac1{\sum_1^n|ze_i|}\sum_{e\in E}|ze|=\frac1{\sum_1^n|ze_i|}\sum_1^n\sum_{e\in E_i}|ze|,
\end{equation}
because the edge $e=[x,y]$ gives a vertex of $P(z,t)$ for all $t \in[zx,zy]$. Equation (\ref{eq:Ave}) is the new expression for $A(P,z).$

\medskip
\section{Proof of Lemma~\ref{l:P-P}}\label{sec:lem11}

We go back to Figure~\ref{fig:path} and the notation there. We assume that $ze>0$ for every $e \in E$ by changing the orientation of $e$ if necessary. Define $\overline {e_i}=\sum_{e \in E_i}e$. Formula (\ref{eq:Ave}) says that
\[
\la\left(\sum_1^nze_i\right)=\sum_1^n\left(\sum_{e\in E_i}ze\right)=z\sum_1^n\overline {e_i}.
\]
This holds for every $z \in \R^d$ satisfying $ze>0$ for every $e\in E$. Consequently
\[
\la \sum_1^n e_i=\sum_1^n \overline {e_i}.
\]

Assume that $e_1'$ (see Figure~\ref{fig:path}) is the edge (possibly virtual) opposite to $e_1$, and choose a vector $z^*$ such that $z^*e>0$ for every $e\in E_i$, $i>1$ and $z^*e<0$ for every $e \in E_1$. Then we have the path $e_2,e_3,\ldots,e_n,e_1'$ from $v'$ to $u'$ (see Figure~\ref{fig:path}) and by (\ref{eq:Ave})
\[
\la \left(\sum_2^nz^*e-z^*e_1'\right)=\sum_2^nz^*\overline {e_i}-z^*\overline {e_1},
\]
implying again that $\la\left(\sum_2^n e_i-e_1'\right)=\sum_2^n\overline {e_i}-\overline {e_1}.$ This immediately gives $\la(e_1+e_1')=2\overline {e_1}$, meaning that for a pair of opposite edges $e_1,e_1'$ the sum $e_1+e_1'$ equals $\frac 2{\la}\overline {e_1}$, the same vector which is exactly an edge in $P-P$. Analogously $e_i+e_i'$ equals $\frac 2{\la}\overline {e_i}$ for every pair of opposite edges $e_i,e_i'\in E_i$ (for every $E_i$). This shows that $P-P$ is indeed a zonotope generated by the vectors $\frac 2{\la}\overline {e_i}$, $i =1,\ldots,n$.\qed

\bigskip
{\bf Proof} of Fact~\ref{fact:zono}. We work in $\R^3$ for this proof. Projecting $P$ to the plane orthogonal to the direction of the vectors in $E_i$ we get a centrally symmetric polygon with $2(n-1)$ edges. Thus $|E_i|=2(n-1)$. Direct every edge $e\in E$ so that $ze>0$ where $z$ comes from Figure~\ref{fig:path}. Since $P$ is a zonotope, each edge in $E_i$ is the same vector $v_i,$ and each $v_i$ appears as one of the edges on the $e_1,\ldots,e_n$ path in Figure~\ref{fig:path}. Consequently $t^+(z)-t^-(z)=\sum_1^n zv_i$. Equation (\ref{eq:Ave}) shows that
\begin{eqnarray*}
 \nonumber  A(P,z)&=&\frac1{\sum_1^n zv_i}\sum_1^n\sum_{e\in E_i}ze\\
      &=&\frac1{\sum_1^n zv_i}\sum_1^n|E_i|zv_i=2(n-1). \qed
\end{eqnarray*}

\medskip
The same argument works in every dimension $d\ge 3$: If the zonotope $P \subset \R^d$ has $n\ge d$ generators and every $d$ of them are linearly independent, then $|E_i|=2{n-1 \choose d-2}$ and $A(P,z)=2{n-1 \choose d-2}$, indeed a constant for every $z\ne 0$. For more information on zonotopes and their relations to hyperplane arrangements see the books by Schneider~\cite{schn} and by Ziegler~\cite{zieg}.

\medskip
This proof, combined with the $z,z^*$ argument used in Lemma~\ref{l:P-P} shows as well that, $A(P,z)=\la$ is a constant for a zonotope $P \subset \R^d$ if and only if $|E_i|=\la$ for every $i.$

\bigskip
\section{ Non-zonotopes with $A(P,z)$ constant}\label{sec:exampl}

It is evident that a convex polygon $P$ is always a half-zonotope, and $A(P,z)=2$ for every $z$. For higher dimensions the key observation is the following.

\begin{lemma}\label{l:key} 	If $P$ and $Q$ are half-zonotopes in $\R^d$, then so is $P+Q.$
\end{lemma}

The {\bf proof} is simple: $(P+Q)-(P+Q)=\{p+q-p'-q':p,p'\in P, q,q'\in Q\}=(P-P)+(Q-Q).$ \qed

\medskip
It follows that a half-zonotope need not be a zonotope, for instance if $Q_1$ and $Q_2$ are two-dimensional polygons in $\R^3$ lying in non-parallel planes, then $P=Q_1+Q_2$ is a half-zonotope but is not a zonotope. A concrete example is when both $Q_1$ and $Q_2$ are triangles, see Figure~\ref{fig:non-zono}, where $Q_1$ is the black and $Q_2$ is the red triangle, and the edges of $P$ are drawn with heavy segments. There are several similar examples, for instance when $Q_1$ is a convex polygon and $Q_2$ is a parallelogram. In these cases $A(Q_1+Q_2,z)$ is not a constant as one can check directly.

\begin{figure}[h]
\centering
\includegraphics[scale=0.9]{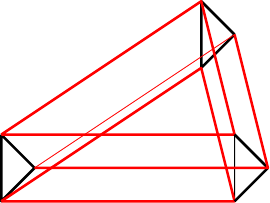}
\caption{A half-zonotope which is not a zonotope.}
\label{fig:non-zono}
\end{figure}

We note further that under the above conditions and with notation $P=Q_1+Q_2$
\[
w(P,z)=w(Q_1,z)+w(Q_2,z) \mbox{ when } z \mbox{ is a unit vector}.
\]
It is evident that $w(Q_1,z)=\frac 12\sum_1^k|zu_i|$ where $u_1,\ldots,u_k$ are the edge vectors of the polygon $Q_1$. Thus in this case the computation of $t^+(P)-t^-(P)$ is easy even if $z$ is not a unit vector:
\begin{equation}\label{eq:width}
t^+(P)-t^-(P)=\frac 12 \sum |uz|
\end{equation}
where the summation goes for all edges $u$ of $Q_1$ and of $Q_2$.

\medskip
Next we give a few examples where $P$ is not a zonotope but $A(P,z)$ is a constant.

\smallskip
{\bf Example 5.1} Let $Q_1$ and $Q_2$ be two convex quadrilaterals, lying in non-parallel planes in $\R^3$ and set $P=Q_1+Q_2$, see Figure~\ref{fig:halfzono}, where $Q_1$ is the black and $Q_2$ is the red quadrilateral. Again, the edges of $P$ are drawn with heavy segments. In two copies of $Q_1$ two edges (of $Q_1$) are not edges of $P$, they are drawn with thin segments. Same applies to $Q_2.$ One can check that each edge of $Q_1$ and $Q_2$ appears as an edge of $P$ exactly three times. With the notation of Section~\ref{sec:ave}, $|E_i|=3$ for every one of the 8 classes of the edges of $P$. The polytope $P$ is not a zonotope. Its edge vectors in each class are the same, say $v_i$ in $E_i.$ In view of equation (\ref{eq:width}), $t^+-t^-=\frac 12\sum_{i=1}^8|zv_i|$ which equals  $\frac 12\sum_{i=1}^8zv_i$ when the orientations are chosen to satisfy $zv_i>0$ for every $i.$ Then the suitably modified version of equation (\ref{eq:Ave}) applies and shows that
\[
A(P,z)=\frac{2}{\sum_1^8 zv_i}\sum_1^8|E_i|zv_i=6.
\]

\begin{figure}[h]
\centering
\includegraphics[scale=0.85]{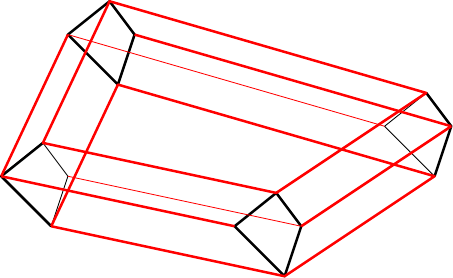}
\caption{A non-zonotope with $A(P,z)=6$.}
\label{fig:halfzono}
\end{figure}
We'd like to mention that this example is a bit of a cheating because it is almost the same as a zonotope with 4 generators. There we get 4 classes of edges, and 6 (identical as vectors) edges in each class. Here the opposite sides of $Q_i$ are not parallel, so we have 8 classes with 3 edges in each.

\medskip
{\bf Example 5.2} Essentially the same method works and gives a half-zonotope $P$ (which is not a zonotope) with $A(P,z)=2(k+1)$ for every integer $k\ge 2.$ Just take two slightly perturbed copies of the regular $2k$-gon for $Q_1$ and $Q_2$, making sure that opposite vertices remain opposite after the perturbation. Then $P=Q_1+Q_2$ is a half-zonotope because of Lemma~\ref{l:key}, and there are $2k+2k$ classes of edges, each class $E_i$ containing $k+1$ copies of the same edge, represented by the vector $v_i$. Again, $t^+-t^-=\frac 12\sum_{i=1}^{4k}zv_i$ assuming that $zv_i>0$ for every $i$. The modified version of equation (\ref{eq:Ave}) shows then that $A(P,z)=2(k+1).$ This example is also similar to the case of a zonotope with $k+k$ generators where the first $k$ generators are coplanar, and so are the next $k$ ones.

\medskip
{\bf Example 5.3} in which there is no cheating. Write $f_1,f_2,f_3$ for the standard basis vectors of $\R^3$ and let $T_1$ be the triangle with vertices $0,f_2,f_3$. Similarly the triangles $T_2$ and $T_3$ have vertices $0,f_1,f_3$ and $0,f_1,f_2$, respectively. The polytope $P=T_1+T_2+T_3$ is a half-zonotope but not a zonotope yet $A(P,z)=4$ for every $z$ as one can check directly. Figure~\ref{fig:nocheat} shows the non-zonotope $P$, its facets are 3 triangles, 3 pentagons, and one hexagon (coloured blue in the figure).
\begin{figure}[h]
\centering
\includegraphics[scale=0.7]{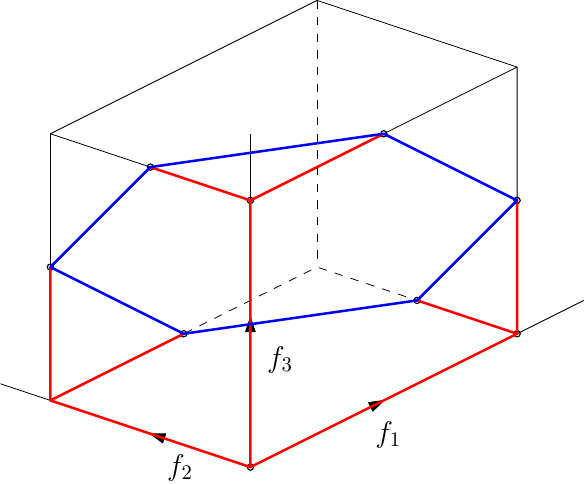}
\caption{A non-zonotope with $A(P,z)=4$.}
\label{fig:nocheat}
\end{figure}
\medskip

There are several similar examples in $\R^3$, and also in higher dimensions. For instance in $\R^4$, take two convex $k$-gons (or even two triangles)  $Q_1,Q_2$ in two general position planes. Their sum $P=Q_1+Q_2$ is not a zonotope but $A(P,z)=2k$ for every $z$.

\bigskip
\section{Proof of Theorem~\ref{th:main}}\label{sec:main}

Assume $P\subset \R^3$ is a centrally symmetric polytope with $A(P,z)$ constant, say $\la$. According to Lemma~\ref{l:P-P} $P$ is a half-zonotope. As it is centrally symmetric, Fact~\ref{f:o-sym} shows that $P$ itself is a zonotope. Assume its generator set is $V=\{v_1,\ldots,v_n\}$. The hypergraph $\F$ was defined in Section~\ref{sec:introd}. So it suffices to prove

\begin{claim}\label{cl:deg} Under the above conditions $\la=2\deg v$ for every $v \in V$.
\end{claim}

{\bf Proof.}  Projecting $P$ to the plane orthogonal to $v_i$ gives a centrally symmetric convex polygon with $2\deg v_i$ edges (and vertices) as each edge of this polygon corresponds to an edge $W \in \F$ that contains $v_i$. Thus $|E_i|=2\deg v_i$.

Orient each edge $e\in E_i$ the same way as $v_i$ is oriented. Then $\overline {e_i}=|E_i|v_i=2(\deg v_i)v_i$. If $e_i,e_i'\in E_i$ are opposite edges then $e_i+e_i'=2v_i$ because $P$ is a zonotope.
The proof of Lemma~\ref{l:P-P} shows that$\la (e_i+e_i')=2\overline {e_i}$. Consequently $2\la v_i=\la (e_i+e_i')=4(\deg v_i)v_i$. \qed

\medskip
The same argument works for the proof of Theorem~\ref{th:maind}, we omit the details.

\medskip
The claim shows that, for a zonotope $P\subset \R^3$, $A(P,z)$ is a constant if and only if $\deg v$ is the same number for every $v \in V$. In particular, the lengths of the generators do not matter, only the degrees count. So one can identify a generator $v\in V$ with the line $L(v)=\{\al v: \al \in \R\}$ and consider a plane $S$ (not containing the origin and not parallel with any $v \in V$), and further identify $v\in V$ with the point $v^*=L(v)\cap S$. This way we have a new representation of  $\F$, to be called the $S$-{\it representation}: the vertex set is $V^*=\{v^*: v\in V\}$, and the edges of $\F$ are formed by sets of collinear points of $V^*$. There are several examples of zonotopes with $A(P,z)$ a constant.

\bigskip
{\bf Example 6.1} is given in Fact~\ref{fact:zono} where $\deg v_i=n-1$. $V^*$ is just $n$ points in general position (no three on a line) in the plane $S$.

\begin{figure}[h]
\centering
\includegraphics[scale=0.8]{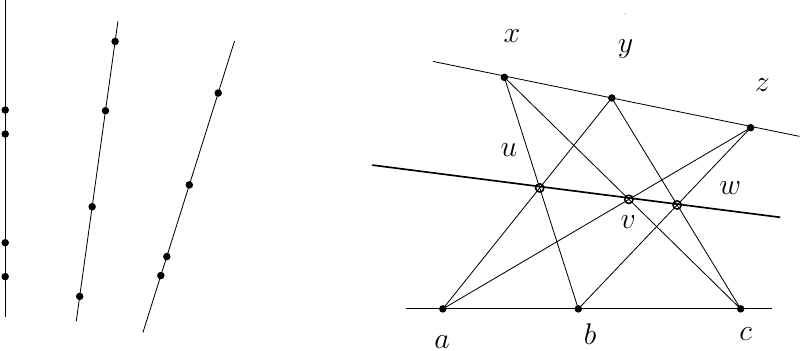}
\caption{Example 6.2 with $\ell=3$, $k=4$ and Example 6.3.}
\label{fig:3-4}
\end{figure}

\medskip
{\bf Example 6.2} is again explained in the $S$-representation of $\F$. We have $\ell$ lines in $S$ and on each line $k$ points and $V^*$ consists of these $\ell k$ points, everything else is in general position. Then $\deg v=1+(\ell-1)k$ for every $v \in V$. This gives a zonotope $P$ with $n=\ell k$ generators and $A(P,z)=2(n-k+1)$. The hypergraph $\F$ consists of $\ell$ disjoint $k$-tuples (corresponding to the $\ell$ lines in $S$) and of all pairs of points that come form distinct $k$-tuples (or lines). Figure~\ref{fig:3-4}, left shows the points of $V^*$ when $\ell=3$ and $k=4$.

\medskip
{\bf Example 6.3.} For two points $p,q\in S$ let $L(p,q)$ denote the line connecting them. The example is three points, $a,b,c$ on a line and three further points, $x,y,z$ on another line, and three more points, namely $u=L(a,y)\cap L(b,x),v=L(a,z)\cap L(c,x)$ and $w=L(b,z)\cap L(c,y)$, see Figure~\ref{fig:3-4} right. $V^*$ consists of these nine points. By the Pappos theorem $u,v,w$ are also collinear. In $\F$ there are nine collinear triples and nine pairs (namely $ax, aw, by, bv, cu, cz, xw, yv, zu$), and $\deg v=5$ for every $v \in V$.

\medskip
In the $S$-representation we have a finite set of points in the plane $S$ that define lines and also the hypergraph $\F$. It is well known that there are {\sl ordinary lines}, that is, pairs of points $p,q \in V^*$ such that $L(p,q)$ contains no further point from $V^*$. This was a question of Sylvester from 1893, solved fifty years later by Gallai (see \cite{Erd} and \cite{Erd82}) and Melchior \cite{Mel}. Around 1960 B\"or\"oczky (unpublished but see for instance \cite{CrMc} or  \cite{GreTao}) constructed examples of $n$ points in the plane with a small number of ordinary lines.
According to a famous result of Green and Tao \cite{GreTao} these examples give the minimal number of ordinary lines for $n$ points.  Interestingly, in these examples $\deg v$ is not a constant, so $A(P,z)$ is not a constant for the corresponding zonotope.

\bigskip
\section{A related average}\label{sec:tiles}

The standard tiling of $\R^3$ with unit cubes is the collection of cubes $C(a)=\{(x_1,x_2,x_3)\in \R^3: a_i\le x_i \le a_{i+1}, i=1,2,3\}$ where $a=(a_1,a_2,a_3)$ is a lattice point in $\R^3$, that is, every $a_i$ is an integer. For a 2-dimensional plane $L\subset \R^3$ the polygons $L\cap C(a)$ (for the cubes when this intersection is non-empty) define a tiling $\T(L)$ of $L$. Clearly each tile $T$ of $\T(L)$ is a convex polygon with $e(T)=3,4,5$ or $6$ edges (or vertices). What is the average number $\Ave(L)$ of $e(T)$?

As $\T(L)$ is infinite, this average is to be taken with some caution. Assuming $Q$ is a (large) cube in $\R^3$, the average of $e(T)=v(T)$ for $T\in \T(L),T \subset Q$ is given as
\[
\Ave(L,C)=\frac{\sum_{T\in \T(L),T \subset Q}v(T)}{|\{T\in \T(L),T \subset Q\}|}.
\]
To define $\Ave(L)$ take a sequence of large cubes $Q_1,Q_2,\ldots$ with their diameter tending to infinity and set $\Ave(L)=\lim \Ave(L,Q_n)$. Standard arguments show that the limit exists and is independent of the choice of the sequence $Q_1,Q_2,\ldots$.

\begin{theorem}If $L$ contains no lattice point, then $\rm{Ave}(L)=4$.
\end{theorem}

{\bf Proof,} only a sketch. We identify the plane $x_3=0$ with $\R^2$ and assume that $L$ is not orthogonal to this plane. The projection of a tile $T \in \T(L)$ to $\R^2$ is denoted by $T^*$. Then $\T^*:=\{T^*:T \in \T(L)\}$ is a tiling of $\R^2$, and the tiles in $\T^*$ are determined by the lines $\ell_{a_1}:=\{x_1=a_1\}, \ell_{a_2}:=\{x_2=a_2\}$ and $\ell_{a_3}$ which is the projection to $\R^2$ of the line $L\cap\{x_3=a_3\}$, of course $a_1,a_2,a_3$ are integers. The vertices of $\T^*$ are formed by the intersections $\ell_{a_i}\cap \ell_{a_j}$, for distinct $i,j \in \{1,2,3\}.$

\smallskip
Let $Q$ be one of the cubes in the sequence $Q_1,Q_2,\ldots$ with $\diam Q=m$, $m$ large. It is not hard to check (we omit the details) that the number of vertices of $\T^*$ lying in $Q$ is $cm^2+O(m)$ where $c>1$ is a constant explicitly computable from the parameters of $L$.

\smallskip
This implies that the number $n(T^*)$ of tiles in $\T^*$ contained in $Q$ is also $cm^2+O(m)$. This is done by well-known argument: let $b\in \R^2$ be a vector such that  the linear functional $bx$ takes distinct values on the vertices of $\T^*$, and associate with each tile $T^*\in \T^*$ its vertex where the functional $bx$ takes its minimal value on $T^*$. These vertices are all distinct, so  $n(T^*)\le cm^2+O(m)$. The opposite inequality follows from the fact that there are only $O(m)$ tiles that intersect $Q$ and are not contained in $Q$.

\smallskip
We are almost finished. Each vertex in $\T^*$ is adjacent to four edges in the tiling (because $L$ contains no lattice point). The edges are counted twice by their two endpoints, apart from a few, at most $O(m)$, boundary edges, and each edge appears in two tiles (again apart from a few boundary edges). Thus the total number of edges is $4n(\T^*)+ O(m)$ and the ratio defining $\Ave(L,C)$ tends to 4 as $m \to \infty.$

\medskip
The condition ``$L$ contains no lattice point" is important, because for instance $\Ave (L)=3$ for the plane $L$ defined by the equation $x_1+x_2-x_3=0$: all tiles in $\T(L)$ are triangles. We mention further that the above argument extends to higher dimensions.

\bigskip
\section{Motivation from geology}\label{sec:Mot}
 \subsection{Primary fracture and the cube}
The geometry of fractured rock is in the forefront of interest in geology \cite{adler_fracture_network_book}. A recent study \cite{Plato} showed that
a large portion of so-called \emph{primary crack patterns} can be very well approximated by \emph{hyperplane mosaics} which are space-filling, convex tessellations generated by
hyperplanes in random positions \cite{grunbaumshepard, schneider2008stochastic, senechal}. There is one striking feature of  $d$-dimensional hyperplane mosaics which does not depend on the specific distribution (to which we will refer as the \emph{primary distribution}) generating the random positions of hyperplanes: the average values of combinatorial features
(e.g. numbers of faces, edges and vertices \cite{domokos2019honeycomb}) of the convex polyhedra
(to which we will refer to  as \emph{primary fragments)} agree with the respective values corresponding to the $d$-dimensional cube \cite{schneider2008stochastic}.

Rock fragments are the result of a multi-level fracture process \cite{adler_fracture_network_book, bohn2005hierarchical, disc}: primary (global) fracture is followed by secondary (local) fracture.
In secondary fracture, individual primary fragments are locally bisected by planes picked from a \emph{secondary random distribution} and \emph{secondary fragments} are created in this process.
Secondary fracture can also be viewed as a recursion which we introduce below.

\subsection{Secondary fracture interpreted as a recursion}

\subsubsection{The general case}

Let $P_{0,1}$ be a $d$-dimensional convex polytope with $V_{0,1}$ vertices and let us consider a cut by
a  hyperplane  $H_{0,1}$ which intersects $P_{0,1}$ in a generic manner to create one
$(d-1)$-dimensional convex polytope  $P^{\star}_{0,1}$ with $k^{\star}$ vertices.
$P^{\star}_{0,1}$ separates $P_{0,1}$ into two convex, $d$-dimensional \emph{descendant} polytopes
$P_{1,1}, P_{1,2}$ with respective vertex numbers $V_{1,1},V_{1,2}$. We introduce the notation $P_0=\{P_{0,1}\}, \quad P_1 = \{P_{1,1},P_{1,2}\}$
for the set of polytopes in steps $0$ and $1$, respectively and the notation $H_0=\{H_{0,1}\}$ for the
set of hyperplanes in step 0.  The average number of vertices of  $P_{1,1},P_{1,2}$ we denote by
$V_1=(V_{1,1}+V_{1,2})/2$. Now we can write
\begin{equation}\label{eq:rec01}
\begin{array}{rcl}
P_1 & = & f(P_0, H_0) \\
V_1 & = & g(P_1),
\end{array}
\end{equation}
where the function $f$, operating sets of polytopes and corresponding hyperplanes, is the binary cut described above,
and the function $g$, operating on sets of polytopes, is counting the average number of vertices in the given set.

Our aim is to  generalize (\ref{eq:rec01}) to a recursion formula.  However, before doing so, we introduce some related concepts.
\begin{defn}\label{def:ccuts}
We call the hyperplane cut $H_{0,1}$  \emph{critical},(\emph{supercritical, subcritical}) if $V_1=V_0$ ($V_1>V_0, V_1<V_0$).
\end{defn}
We can write the following simple
\begin{lemma}
The hyperplane cut $H_{0,1}$ is critical (supercritical, subcritical) if and only if $k^{\star}=\frac{V_0}{2}$ ($k^{\star}>\frac{V_0}{2}, k^{\star}<\frac{V_0}{2}$).
\end{lemma}

\begin{proof}
Let us denote the number of vertices of $P_{0,1}$ which are also vertices of $P_{1,i}$
by $V_{0,i}$. Then we can write:

\begin{equation}\label{cut1}
\begin{array}{rcl}
V_{0,1} + V_{0,2}  & = & V_0 \\
V_{1,1}  & = & V_{0,1} + k^{\star} \\
V_{1,2}  & = & V_{0,2} + k^{\star} = V_0-V_{0,1}+k^{\star}.\\
\end{array}
\end{equation}

Based on the above, for the average vertex number $V_1=(V_{1,1}+V_{1,2})/2$ of the descendants we have:

\begin{equation}\label{generalcut}
    V_1 = \frac{V_0}{2}+k^{\star}.
\end{equation}
Formula (\ref{generalcut}) proves the statement of the Lemma.
\end{proof}

In the next step we can generalize equation \ref{eq:rec01} to
\begin{equation}\label{eq:rec02}
\begin{array}{rcl}
P_{i} & = & f(P_{i-1}, H_{i-1}) \\
V_{i} & = & g(P_{i}),
\end{array}
\end{equation}
where $P_i$ is the set of descendant polytopes after $i$ steps and $P_i$ has $2^i$ elements. Similarly,
$H_i$ is the set of $2^i$ hyperplanes bisecting the corresponding elements of the set $P_i$.
We develop equation (\ref{eq:rec02}) into a recursion formula for the sequence $P_i$ (which also defines the sequence $V_i$) using a stochastic model, which is another way of defining $A(P,z)$.

\smallskip
Let $z$ be a unit vector. Recall the definition of $t^{\pm}=t^{\pm}(P,z)$. Consider the hyperplane $\{x \in \R^d: zx=t\}$ as a random hyperplane $H_{0,1}$ where $t$ is a random and uniform element of the interval $[t^-,t^+]$. Then $V_1$ is a random variable. We will denote its
expected value by $ V_1(z)$. Based on equations (\ref{eq:ave}) and (\ref{generalcut}), we have
\begin{equation}\label{randomcut1}
V_1(z)=\frac{V_0}2+A(P_{0,1},z).
\end{equation}

Next we let $z$ be selected uniformly randomly on the sphere.
\begin{defn}\label{cpolytope}
We will denote the expected value of $V_1(z)$ by $\bar V_1$ and we will call a polytope \emph{critical} (supercritical, subcritical) if $\bar V _1=V_0$ ($\bar V _1>V_0$, $\bar V _1<V_0$).
\end{defn}

Using this definition and formula (\ref{eq:rec02}), now we can write
\begin{equation}\label{eq:rec03}
\begin{array}{rcl}
P_{i} & = & \bar f(P_{i-1}) \\
V_{i} & = & g(P_{i}),
\end{array}
\end{equation}
where the function $\bar f$, operating on sets of polytopes, is the random binary cut described in Definition \ref{cpolytope}.
We can see that equation (\ref{eq:rec03}) defines the sequence $V_i$ as a projection of a direct recursion (defining the sequence $P_i$
of polytopes, with set $P_i$ containing $2^i$ polytopes). Now we may
ask about the convergence properties of $V_i$. More precisely, we call a value $V_i=V^{\star}$ \emph{weakly critical}
if $V_{i+1}=V_i$ and we are interested whether such weakly critical values may exist because
such weakly critical values could become dominant in experimental data.
 In general, this question is very difficult as it depends on the
average vertex number of a collection of $2^i$ descendant polytopes.

\smallskip
We call the set $P_i$ \emph{uniform} if all $2^i$ polytopes in the set $P_i$
are identical. (This is, in essence equivalent of executing the first step on a single polytope $2^i$ times
and ask for the time average.) We are interested in the existence of critical polyhedra because
if such shapes exist then, in the uniform case, the sequence $V_i$ will be stationary, at least for one step.

\smallskip
If $A(P_{0,1},z)=\lambda$ (i.e. it does not depend on $z$) then, based on equation (\ref{randomcut1}),  we have
\begin{equation}
\bar V_1=\frac{V_0}2+\lambda
\end{equation}
and we can see that the initial polyhedron $P_{0,1}$ will be critical if

\begin{equation}\label{critical1}
\lambda=V_0/2.
\end{equation}

\subsubsection{The 2D case}

In $d=2$ dimensions, the polytope  $P^{\star}_{0,1}$ is always a finite line segment so we have $k^{\star}=2$ and
(\ref{generalcut}) translates into
\begin{equation}\label{2Dcut}
    V_1 = \frac{V_0}{2}+2,
\end{equation}
so we can see that only quadrangles can be critical polygons and it is easy to show that
parallelograms are indeed critical.

\subsubsection{The 3D case}
In $d=3$ dimensions, the polytope $P^{\star}_{0,1}$ is a 2D polygon and can have any number of vertices. Despite this apparent broad ambiguity, for a class of secondary distributions computer experiments showed \cite{Plato} that starting with a cube, the average $\bar V=8$ vertices  remained a
good approximation of the computed averages of secondary fragments. This computational observation showed a very good match with field and laboratory measurements of fragments which were the result of successive steps of primary and secondary fragmentation.

\smallskip
Since we know very little about the recursion (\ref{eq:rec03}), the full mathematical explanation of these computational result is lacking. However,  in this paper we showed that in 3D parallelepipeds are critical polytopes. We can also prove that they are the only critical polytopes that are centrally symmetric with $A(P,z)=\la$, a constant. Perhaps they are the only critical polytopes.
Our results suggest that the $\bar V=8$ average observed in the computer experiment may indeed play a central role in secondary fragmentation, for a broad range of secondary distributions.

\bigskip
{\bf Acknowledgements.} The support of the HUN-REN Research Network is appreciated.  The first author was partially supported from NKFIH grants No 131529, 132696, and 133919.  Support for the second author from NKFIH grant 134199 and of grant BME FIKP-V\'AZ by EMMI is kindly acknowledged.

\vskip0.5cm

\noindent
Imre B\'ar\'any \\
Alfr\'ed R\'enyi Institute of Mathematics, HUN-REN\\
13 Re\'altanoda Street, Budapest 1053 Hungary\\
{\tt barany.imre@renyi.mta.hu}\\
and \\
Department of Mathematics\\
University College London\\
Gower Street, London, WC1E 6BT, UK\\

\noindent
G\'abor Domokos\\
Department of Morphology and Geometric Modeling, and \\
HUN-REN-BME Morphodynamics Research Group\\
Budapest University of Technology and Economics\\
M\H uegyetem rkp 3, Budapest, 1111 Hungary\\
{\tt domokos@iit.bme.hu}

\end{document}